\documentclass[a4paper,11pt]{article}
\usepackage[T1]{fontenc}
\usepackage[utf8]{inputenc}
\usepackage[english]{babel}
\usepackage{lmodern,amsmath,amssymb,amsthm}
\usepackage[margin=28mm]{geometry}
\usepackage{microtype,xurl}
\usepackage[hidelinks]{hyperref}
\usepackage{enumitem}
\hypersetup{pdftitle={On a characterization of spaces satisfying open mapping and equivalent theorems},pdfauthor={Henning Wunderlich}}
\newcommand{\K}{\mathbb K}
\newcommand{\N}{\mathbb N}
\newcommand{\CC}{\mathcal C}
\newcommand{\BP}{\mathsf P}
\newcommand{\BR}{\mathsf R}
\newcommand{\Graph}{\operatorname{Graph}}
\newcommand{\Ptak}{Pt\'ak}
\newcommand{\Frec}{Fr\'echet}
\newtheorem{theorem}{Theorem}[section]
\newtheorem{proposition}[theorem]{Proposition}
\newtheorem{lemma}[theorem]{Lemma}
\newtheorem{corollary}[theorem]{Corollary}
\theoremstyle{definition}
\newtheorem{definition}[theorem]{Definition}
\newtheorem{example}[theorem]{Example}
\theoremstyle{remark}
\newtheorem{remark}[theorem]{Remark}
\numberwithin{equation}{section}
\setlist[enumerate]{label=\textup{(\roman*)},leftmargin=2em,itemsep=2pt,topsep=4pt}
\title{On a characterization of spaces satisfying open mapping and equivalent theorems}
\author{Henning Wunderlich\thanks{Offenbach, Germany.
E-mail: \href{mailto:HenningWunderlich@t-online.de}{\texttt{HenningWunderlich@t-online.de}}}}
\date{\today}

\usepackage{enumitem,microtype,booktabs,array}
\usepackage{xcolor}

\newcommand{\doi}[1]{\href{https://doi.org/#1}{doi:\nolinkurl{#1}}}
\newcommand{\arxiv}[1]{\href{https://arxiv.org/abs/#1}{arXiv:\nolinkurl{#1}}}

\begin{document}
\maketitle
\begin{abstract}
We study open-mapping, continuous-inverse and closed-graph properties for
maps whose domain and range belong to the same class of Hausdorff
topological vector spaces. Quotient and graph stability imply equivalence
of these properties. In the locally convex setting, we characterize
barrelled infra-\Ptak\ spaces and barrelled \Ptak\ spaces as largest
classes determined by the closed-graph property, Banach test spaces and
appropriate nearly-open image conditions.
\end{abstract}

\section{Introduction and scope}
The open-mapping and closed-graph theorems extend beyond Banach spaces
through completeness conditions introduced by \Ptak\ and the
closed-graph theorem of the Robertsons \cite{Ptak58,RR56}.
Mahowald's converse characterizes barrelled spaces by closed-graph maps
into Banach spaces \cite{Mahowald61}. We combine these classical results
when the domain and range are required to belong to one class.

Class-relative mapping theorems have substantial antecedents. Husain
studied $B(\mathcal C)$-spaces \cite{Husain67}; Iyahen's categorical
formulation includes an image-stability characterization of $B$- and
$B_r$-completeness \cite[Proposition 2.3]{Iyahen68}; Kalton treated
domain classes determined by prescribed range classes \cite{Kalton71}.
Standard accounts are \cite{Kothe79,SW99}. The present purpose is a
concise largest-class formulation, with the roles of quotient, graph
and nearly-open image stability stated separately. 

Alongside the classification results, we give
equivalent rigidity formulations, a separating-family continuity
criterion and a class-relative formulation for closed linear relations.
The last viewpoint has classical antecedents in Persson
\cite[Theorem 2]{Persson66} and is developed for barrelled and \Ptak\
spaces by Robinson \cite{Robinson19}. These additions are short
reformulations and applications of classical mapping theorems; no
claim of novelty is made for them. Examples delimit the hypotheses.

\section{Three mapping properties and their closure hypotheses}
Throughout, the scalar field is fixed as $\K=\mathbb R$ or $\mathbb C$.
All spaces are Hausdorff; images have their subspace topology, quotients
their quotient topology, and graphs their induced product topology.
Local convexity is not required in this section. Classes are invariant
under topological linear isomorphisms.
For $u:E\to F$ write $\Graph(u)=\{(x,u(x)):x\in E\}$.

\begin{lemma}\label{lem:graph}
If $u:X\to Y$ is continuous and $Y$ is Hausdorff, then its graph is closed.
\end{lemma}
\begin{proof}
The graph is the inverse image of the closed diagonal of $Y\times Y$
under the continuous map $(x,y)\mapsto(u(x),y)$.
\end{proof}

\begin{definition}
For a class $\CC$ of Hausdorff topological vector spaces let:
\begin{description}[leftmargin=3em,labelwidth=2.4em]
\item[\textup{(O)}] Every continuous linear surjection $E\to F$ with
$E,F\in\CC$ is open.
\item[\textup{(C)}] Every continuous linear bijection $E\to F$ with
$E,F\in\CC$ has continuous inverse.
\item[\textup{(G)}] Every linear map $E\to F$ with $E,F\in\CC$ that
has closed graph is continuous.
\item[\textup{(Q)}] If $E\in\CC$ and $N$ is a closed linear subspace,
then $E/N\in\CC$.
\item[\textup{(H)}] If $E,F\in\CC$ and $u:E\to F$ is linear with
closed graph, then $\Graph(u)\in\CC$.
\end{description}
\end{definition}
By Lemma~\ref{lem:graph}, (G) is equivalent to continuity if and only
if the graph is closed. Since ordered pairs of objects are arbitrary,
(C) is equivalently the symmetric continuous-inverse formulation.

\begin{theorem}\label{thm:equivalence}
For every class $\CC$,
\[
 (O)\Longrightarrow(C),\qquad (G)\Longrightarrow(C).
\]
Under (H), (C) implies (G); under (Q), (C) implies (O).
Consequently, (Q) and (H) together imply equivalence of (O), (C) and (G).
\end{theorem}
\begin{proof}
A bijective open map has continuous inverse, so (O) implies (C).
If $u:E\to F$ is a continuous linear bijection, its graph is closed.
Interchanging coordinates gives the closed graph of $u^{-1}$.
Thus, (G) also implies (C).

Assume (H) and (C). For a linear $u:E\to F$ with closed graph,
the restricted projection $p_E:\Graph(u)\to E$ is a continuous
linear bijection between objects of $\CC$. Its inverse
$v(x)=(x,u(x))$ is continuous by (C), and $u=p_F\circ v$ is continuous.

Assume (Q) and (C). For a continuous linear surjection $u:E\to F$,
put $N=\ker u$. The induced map $u_0:E/N\to F$ is a continuous
linear bijection, hence an isomorphism by (C). The quotient map
$q:E\to E/N$ is open, and $u=u_0\circ q$ is open.
\end{proof}

The precise dependence on the closure assumptions is therefore
\[
\begin{array}{c|c}
\text{implication}&\text{extra assumption}\\ \hline
(O)\Rightarrow(C),\ (G)\Rightarrow(C)&\text{none}\\
(C)\Rightarrow(G)&(H)\\
(C)\Rightarrow(O)&(Q).
\end{array}
\]
In particular (H) alone suffices for $(C)\Longleftrightarrow(G)$.
Closure under finite products and closed subspaces implies (H), but
the graph argument does not require either stronger assumption.

\begin{remark}[Complete spaces and quotients]\label{rem:complete}
Banach spaces and \Frec\ spaces satisfy (Q) and (H). The class of
all complete locally convex spaces does not satisfy (Q).
Dierolf \cite{Dierolf75} proves that every Hausdorff locally convex
space is a quotient of a complete locally convex space. Choosing an
incomplete target disproves quotient stability. Thus, completeness
alone cannot justify an application of Theorem~\ref{thm:equivalence}.
\end{remark}

\section{Completeness and nearly-open images}
From now on all spaces are locally convex. Let $E'$ be the continuous
dual, $\sigma(E',E)$ its weak-star topology, and $E'_\beta$ the strong
dual. A \emph{barrel} is a closed, absolutely convex, absorbing set.
A space is \emph{barrelled} if every barrel is a zero-neighbourhood;
equivalently, every weak-star bounded subset of its dual is
equicontinuous \cite{SW99}.

\begin{definition}
A linear subspace $L\subseteq E'$ is \emph{nearly closed} if
$L\cap A$ is relatively weak-star closed in every equicontinuous
subset $A\subseteq E'$. The space $E$ is \emph{\Ptak}, or
$B$-complete, if every nearly closed linear subspace is weak-star
closed. It is \emph{infra-\Ptak}, or $B_r$-complete, if every
weak-star dense nearly closed linear subspace equals $E'$.
\end{definition}

A linear map $u:E\to F$ is \emph{nearly open onto its image} if
$\overline{u(U)}^{\,u(E)}$ is a zero-neighbourhood in $u(E)$ for every
zero-neighbourhood $U$ in $E$. In this paper ``nearly open'' always
uses the image with its subspace topology. A continuous linear map
is a \emph{topological homomorphism} if it is open onto its image,
equivalently if $E/\ker u\to u(E)$ is a topological isomorphism.

\begin{theorem}[Classical mapping facts]\label{thm:classical}
\leavevmode
\begin{enumerate}
\item $E$ is \Ptak\ if and only if every continuous nearly-open
linear map from $E$ to an arbitrary Hausdorff locally convex space
is a topological homomorphism. It is infra-\Ptak\ if and only if
the same assertion holds for injective maps.
\item A linear map with closed graph from a barrelled space into an
infra-\Ptak\ space is continuous.
\item Closed subspaces preserve \Ptak\ and infra-\Ptak\ completeness.
Hausdorff quotients preserve \Ptak\ completeness and barrelledness.
Every \Frec\ space is barrelled and \Ptak.
\end{enumerate}
\end{theorem}
These facts are due to \Ptak\ and the Robertsons
\cite{Ptak58,RR56}; see also \cite[\S34]{Kothe79} and
\cite[IV, \S8]{SW99}. Barrelledness is not, in general, inherited
by closed subspaces.

\begin{lemma}\label{lem:nearly}
\leavevmode
\begin{enumerate}
\item Every linear surjection onto a barrelled locally convex space
is nearly open.
\item A continuous nearly-open linear surjection from a barrelled
space has barrelled range.
\item If $u:E\to F$ is continuous, linear and nearly open, its
induced injection $u_0:E/\ker u\to u(E)$ has the same three properties.
\end{enumerate}
\end{lemma}
\begin{proof}
For (i), the closure of the image of an absolutely convex
zero-neighbourhood is a barrel in the range. For (ii), if $B$ is
a barrel in the range, then $u^{-1}(B)$ is a barrel, hence a
zero-neighbourhood in the domain. Near-openness and surjectivity
make $\overline{u(u^{-1}(B))}=B$ a zero-neighbourhood.
For (iii), use the quotient universal property and the identity
$u_0(V)=u(q^{-1}(V))$ for every quotient zero-neighbourhood $V$.
\end{proof}

\begin{definition}[Image axioms]
For a class $\CC$ of Hausdorff locally convex spaces let:
\begin{description}[leftmargin=3em,labelwidth=2.4em]
\item[\textup{(I)}] If $E\in\CC$, $F$ is any Hausdorff locally
convex space, and $u:E\to F$ is injective, continuous, linear and
nearly open, then $u(E)\in\CC$.
\item[\textup{(A)}] The same assertion holds without injectivity.
\end{description}
\end{definition}
These conditions are not closure under arbitrary continuous images.
Either condition implies invariance under topological isomorphisms.

\begin{lemma}\label{lem:axioms}
One has $(A)\Longleftrightarrow(I)+(Q)$ and
$(G)\Longrightarrow(H)$.
\end{lemma}
\begin{proof}
Applying (A) to open quotient maps proves (Q); (I) is immediate.
Conversely, factor a continuous nearly-open map through its kernel
quotient and apply Lemma~\ref{lem:nearly}(iii), (Q) and (I).
If $u:E\to F$ has closed graph, (G) makes it continuous.
The map $x\mapsto(x,u(x))$ is an isomorphism from $E$ onto its graph.
The standing invariance under topological isomorphisms therefore gives
(H), without using (I). There is no assumption that $E\times F$ itself
belongs to $\CC$.
\end{proof}

\begin{lemma}[Banach test spaces]\label{lem:test}
For a Hausdorff locally convex space $E$, the following are equivalent:
\begin{enumerate}
\item $E$ is barrelled;
\item every closed-graph linear map from $E$ to a Banach space is continuous;
\item for every set $J$, every closed-graph linear map
$E\to\ell^\infty(J)$ is continuous.
\end{enumerate}
\end{lemma}
\begin{proof}
The first implication follows from Theorem~\ref{thm:classical}(ii),
and (ii) implies (iii). For the converse let $A\subseteq E'$ be
weak-star bounded. The evaluation map
\[
 T_A:E\longrightarrow\ell^\infty(A),\qquad T_A(x)=(f(x))_{f\in A}
\]
is defined everywhere and has closed graph: that graph is the
intersection of the closed sets $\{(x,y):y(f)=f(x)\}$, $f\in A$.
By (iii) it is continuous.
Hence, $x\mapsto\sup_{f\in A}|f(x)|$ is continuous, so $A$ is equicontinuous.
The dual characterization of barrelledness applies.
\end{proof}
The equivalence of (i) and (ii) is Mahowald's characterization
\cite{Mahowald61}; the proof records the smaller test family explicitly.

\begin{remark}[Finite-dimensional tests cannot detect barrelledness]
\label{rem:finite-tests}
Every closed-graph linear map $u:E\to V$ into a finite-dimensional
Hausdorff space is continuous, for every Hausdorff locally convex $E$.
Indeed, $\ker u$ is closed, and the Hausdorff quotient $E/\ker u$ is
finite-dimensional. Its unique Hausdorff vector topology makes the
induced map into $V$ continuous. Composing with the quotient map gives
continuity of $u$. Thus, scalar tests, or even all finite-dimensional
Banach tests, cannot replace the test family in Lemma~\ref{lem:test}.
\end{remark}

\section{The two largest-class characterizations}
Write
\[
 \BR=\{E:E\text{ is barrelled and infra-\Ptak}\},\qquad
 \BP=\{E:E\text{ is barrelled and \Ptak}\}.
\]
A \emph{largest} class with specified properties is one that has those
properties and contains every class having them. The proofs below
establish this directly; no maximal-class existence principle is needed.

\begin{theorem}[Barrelled infra-\Ptak\ characterization]\label{thm:infra}
The class $\BR$ is the largest class of Hausdorff locally convex spaces
containing every Banach space and satisfying (I) and (G).
It suffices instead to require the test spaces $\ell^\infty(J)$ for all sets $J$.
\end{theorem}
\begin{proof}
Let $\CC$ have the stated properties. For $E\in\CC$, (G) and
Lemma~\ref{lem:test} show that $E$ is barrelled. If $u:E\to F$
is injective, continuous, linear and nearly open, (I) puts $u(E)$
in $\CC$. Since (G) implies (C), the bijection $u:E\to u(E)$
is an isomorphism. Theorem~\ref{thm:classical}(i) shows that $E$
is infra-\Ptak. Thus, every such $\CC$ is contained in $\BR$.

Conversely, $\BR$ contains every Banach space and satisfies (G)
by Theorem~\ref{thm:classical}(ii). The injective mapping
characterization makes each image in (I) isomorphic to its domain;
therefore $\BR$ satisfies (I).
\end{proof}

\begin{theorem}[Barrelled \Ptak\ characterization]\label{thm:ptak}
The class $\BP$ is the largest class of Hausdorff locally convex spaces
containing every Banach space and satisfying (A) and (G).
Equivalently, replace (A) by (I) and (Q). Again the test family
$\{\ell^\infty(J):J\text{ a set}\}$ suffices.
\end{theorem}
\begin{proof}
For any class $\CC$ with these properties, (G) and
Lemma~\ref{lem:test} make every member barrelled. Let $E\in\CC$
and let $u:E\to F$ be continuous, linear and nearly open.
Lemma~\ref{lem:axioms} gives (Q), so both $E/\ker u$ and $u(E)$
belong to $\CC$. The induced continuous bijection between them
is an isomorphism by (C), which follows from (G).
Thus, $u$ is a topological homomorphism and $E$ is \Ptak.

Conversely, $\BP$ contains every Banach space and satisfies (G).
It satisfies (Q) by classical permanence and (I) by the mapping
characterization. Lemma~\ref{lem:axioms} then gives (A).
\end{proof}

\begin{corollary}
\label{cor:original}
\leavevmode
\begin{enumerate}
\item $\BP$ is the largest class containing all Banach spaces,
satisfying (Q), (H) and (I), and satisfying any one of (O), (C) and (G).
Under those closure assumptions the three mapping properties are equivalent.
\item $\BR$ is the largest class containing all Banach spaces,
satisfying (H) and (I), and satisfying (C), equivalently (G).
\end{enumerate}
In both parts the Banach spaces may be replaced by the test family
of Lemma~\ref{lem:test}.
\end{corollary}
\begin{proof}
Apply Theorem~\ref{thm:equivalence} and then the respective
largest-class theorem. The required graph stability of the
classified classes follows from Lemma~\ref{lem:axioms}.
\end{proof}

Part (ii) asserts (C) and (G), not an additional noninjective
open-mapping conclusion. The infra-\Ptak\ mapping characterization
in Theorem~\ref{thm:classical}(i) requires injectivity. Omitting
that hypothesis is not a valid application of the theorem.

\begin{remark}[Restricted graph tests]
In the continuous-inverse formulations of Corollary~\ref{cor:original},
it suffices initially to require graph stability only for closed-graph
maps $E\to\ell^\infty(J)$. The graph projection and (C) make
those maps continuous; Lemma~\ref{lem:test} proves barrelledness.
The same image arguments then prove the corresponding completeness
condition, which recovers the full closed-graph property.
\end{remark}

\section{Topological rigidity and separating-family tests}
The following formulations make the mapping characterizations
intrinsic to a space and its quotients. They specialize the classical
coarser-topology and quotient criteria
\cite[\S34.3(1)--(2)]{Kothe79} to barrelled spaces.

\begin{proposition}[Rigidity of barrelled topologies]\label{prop:rigidity}
Let $(E,\tau)$ be a barrelled Hausdorff locally convex space.
The following are equivalent:
\begin{enumerate}
\item $(E,\tau)$ is infra-\Ptak;
\item every continuous linear bijection from $(E,\tau)$ onto a
Hausdorff barrelled locally convex space is a topological isomorphism;
\item if $\sigma\subseteq\tau$ is a Hausdorff barrelled locally convex
topology on the same vector space, then $\sigma=\tau$.
\end{enumerate}
\end{proposition}
\begin{proof}
For (i)$\Rightarrow$(ii), a surjection onto a barrelled space is nearly
open by Lemma~\ref{lem:nearly}(i), and the injective mapping
characterization in Theorem~\ref{thm:classical}(i) applies.
The implication (ii)$\Rightarrow$(iii) follows by taking the identity.

For (iii)$\Rightarrow$(i), let $u:E\to F$ be continuous, injective,
linear and nearly open onto its image. Lemma~\ref{lem:nearly}(ii)
makes $u(E)$ barrelled. Transporting its subspace topology to $E$
gives a Hausdorff barrelled locally convex topology $\sigma\subseteq\tau$.
By (iii) these topologies coincide, so $u$ is an isomorphism onto
its image. Theorem~\ref{thm:classical}(i) proves infra-\Ptak ness.
\end{proof}

\begin{corollary}[Quotient rigidity]\label{cor:quotient-rigidity}
A barrelled Hausdorff locally convex space $E$ is \Ptak\ if and only
if, for every closed linear subspace $N\subseteq E$, the quotient
topology on $E/N$ admits no strictly coarser Hausdorff barrelled
locally convex topology.
\end{corollary}
\begin{proof}
If $E$ is \Ptak, each $E/N$ is barrelled \Ptak, and
Proposition~\ref{prop:rigidity} applies. Conversely, the stated
condition and barrelledness of quotients make each $E/N$ infra-\Ptak.
Given a continuous nearly-open linear map $u:E\to F$, put
$N=\ker u$. By Lemma~\ref{lem:nearly}(iii), the induced injection
$u_0:E/N\to u(E)$ is continuous and nearly open. It is an isomorphism
by infra-\Ptak ness of $E/N$. Thus, $u$ is a topological homomorphism,
and Theorem~\ref{thm:classical}(i) gives the \Ptak\ property.
\end{proof}

The barrelledness assumption in Proposition~\ref{prop:rigidity}
is essential to this formulation: a space not assumed barrelled
need not itself belong to the collection of topologies under
consideration. The next continuity criterion is the classical
separating-family form of the closed-graph theorem
\cite[\S34.7(5)]{Kothe79}; we include its short proof.

\begin{proposition}[Continuity detected by separating functionals]
\label{prop:separating-tests}
Let $E$ be barrelled, let $F$ be infra-\Ptak, and let $M\subseteq F'$
separate the points of $F$. An everywhere-defined linear map
$T:E\to F$ is continuous if and only if $m\circ T\in E'$ for every
$m\in M$.
\end{proposition}
\begin{proof}
Necessity is immediate. Conversely, separation by $M$ gives
\[
 \Graph(T)=
 \bigcap_{m\in M}
 \{(x,y)\in E\times F:m(y)=(m\circ T)(x)\}.
\]
Every set on the right is closed under the stated continuity
assumptions. Therefore $\Graph(T)$ is closed, and
Theorem~\ref{thm:classical}(ii) makes $T$ continuous.
\end{proof}
Neither equicontinuity of $M$ nor closure of its linear span is required.

\begin{example}[Point evaluations in smooth-function spaces]
\label{ex:smooth-tests}
Let $\Omega\subseteq\mathbb R^d$ be nonempty and open, with $d\geq1$,
and let $E$ be barrelled. Suppose that
$T:E\to C^\infty(\Omega;\K)$ is an everywhere-defined linear map.
If $x\mapsto(Tx)(t)$ is continuous on $E$ for every $t\in\Omega$,
then $T$ is continuous for the full \Frec\ topology of
$C^\infty(\Omega;\K)$. Indeed, point evaluations are continuous and
separate this \Frec, hence infra-\Ptak, range space, so
Proposition~\ref{prop:separating-tests} applies. No separate tests
on derivative evaluations are needed. The same argument applies
to an everywhere-defined linear map into the Schwartz space
$\mathcal S(\mathbb R^d;\K)$. In each assertion the map is assumed
already to take values in the stated function space; the criterion
does not supply smoothness or decay from merely pointwise data.
\end{example}

\begin{corollary}[Uniqueness from a common separating dual]
\label{cor:common-dual}
Let $\tau_1,\tau_2$ be Hausdorff locally convex topologies on the same
vector space $E$, both making it barrelled and infra-\Ptak.
If
\[
 M=(E,\tau_1)'\cap(E,\tau_2)'
\]
separates points, then $\tau_1=\tau_2$. In particular, comparable
barrelled infra-\Ptak\ topologies on the same vector space coincide.
\end{corollary}
\begin{proof}
Apply Proposition~\ref{prop:separating-tests} to the identity in
each direction, using the family $M$. If the topologies are
comparable, the dual of the coarser Hausdorff locally convex
topology is a common separating family.
\end{proof}

\section{Closed linear relations}
A \emph{linear relation} from $E$ to $F$ is a linear subspace
$L\subseteq E\times F$. Its domain is $p_E(L)$, and it is
\emph{total} when $p_E(L)=E$. Its fibres need not be singletons.
The relation viewpoint unifies the mapping theorems; see
\cite[Theorem 2]{Persson66} and \cite{Robinson19}. We first give the
elementary quotient reduction and then record the corresponding
class-relative equivalence.

\begin{lemma}[Reduction by the multivalued part]\label{lem:relations}
Let $E,F$ be Hausdorff locally convex spaces and let
$L\subseteq E\times F$ be a total closed linear relation. Set
\[
 N=\{y\in F:(0,y)\in L\},\qquad q:F\to F/N.
\]
Then $N$ is closed, and
\[
 T:E\to F/N,\qquad T(x)=q(y)\quad\text{for any }(x,y)\in L,
\]
is a well-defined linear map with closed graph. Moreover, $T$ is
continuous if and only if $p_E:L\to E$ is open.
\end{lemma}
\begin{proof}
Closedness of $N$ follows by restricting $L$ to $\{0\}\times F$.
Any two values over $x$ differ by an element of $N$, which proves
well-definedness. The continuous surjection
$\mathrm{id}_E\times q:E\times F\to E\times(F/N)$ is open and hence
a quotient map. Since
\[
 L=(\mathrm{id}_E\times q)^{-1}\bigl(\Graph(T)\bigr),
\]
the graph of $T$ is closed.

If $T$ is continuous, then, for open $U\subseteq E$ and $V\subseteq F$,
\[
 p_E\bigl(L\cap(U\times V)\bigr)
   =U\cap T^{-1}\bigl(q(V)\bigr)
\]
is open. These sets come from a basis of the relative topology
of $L$, so $p_E:L\to E$ is open. Conversely, if this projection
is open, it is a quotient map. The identity
$T\circ(p_E|_L)=q\circ(p_F|_L)$ then proves continuity of $T$.
\end{proof}

\begin{definition}
A class $\CC$ has property \textup{(L)} if, whenever $E,F\in\CC$
and $L\subseteq E\times F$ is a total closed linear relation, the
projection $p_E:L\to E$ is open. Neither $L$ nor $E\times F$ is
required to belong to $\CC$.
\end{definition}

\begin{theorem}[A relation form of the equivalence theorem]
\label{thm:relation-equivalence}
For every class $\CC$, property \textup{(L)} implies \textup{(G)}.
If $\CC$ also satisfies \textup{(Q)}, then \textup{(G)} and
\textup{(L)} are equivalent. Under \textup{(Q)} and \textup{(H)},
all four properties \textup{(O)}, \textup{(C)}, \textup{(G)} and
\textup{(L)} are equivalent.
\end{theorem}
\begin{proof}
For (L)$\Rightarrow$(G), take $L=\Graph(u)$ for a closed-graph
linear map $u:E\to F$. The continuous bijection $p_E:L\to E$
is open by (L); its inverse is continuous, as is
$u=p_F\circ(p_E|_L)^{-1}$.

Conversely, under (Q) the quotient $F/N$ in
Lemma~\ref{lem:relations} belongs to $\CC$. Property (G) makes the
induced closed-graph map $T:E\to F/N$ continuous, and the lemma
gives (L). The last assertion follows from
Theorem~\ref{thm:equivalence}.
\end{proof}

\begin{corollary}[Barrelled domain and \Ptak\ target]
\label{cor:barrelled-relations}
If $E$ is barrelled, $F$ is \Ptak, and
$L\subseteq E\times F$ is a total closed linear relation, then
$p_E:L\to E$ is open.
\end{corollary}
\begin{proof}
The quotient $F/N$ in Lemma~\ref{lem:relations} is \Ptak, hence
infra-\Ptak. Theorem~\ref{thm:classical}(ii) makes the induced map
$T$ continuous, and the lemma applies.
\end{proof}
This is the classical relation theorem in the barrelled--\Ptak\
setting \cite{Robinson19}. It assumes no barrelledness of $N$ or $L$
and supplies no continuous selection of the fibres.
Taking $L$ to be a graph recovers the closed-graph theorem with
\Ptak\ target. Taking it to be the transpose of the graph of a
surjective closed-graph map $S:F\to E$ gives openness of $S$, since
\[
 S(V)=p_E\bigl(L\cap(E\times V)\bigr)\quad(V\subseteq F\text{ open}).
\]
This last conclusion does not require an initial continuity
assumption on $S$.

\section{Examples delimiting the hypotheses}
\begin{proposition}[Reflexive \Frec\ duals]\label{prop:dual}
If $F$ is reflexive and \Frec, then $F$ and $F'_\beta$ are barrelled \Ptak.
\end{proposition}
\begin{proof}
The assertion for $F$ is classical. Put $G=F'_\beta$; reflexivity
identifies $G'$ with $F$. A weakly bounded subset of $F$ is bounded
in its original topology and hence equicontinuous as a subset of
$G'$, by the definition of the strong topology. Thus, $G$ is barrelled.
If $L\subseteq G'=F$ is nearly closed and $x_n\in L$ converges to $x$
in $F$, then $\{x,x_1,x_2,\ldots\}$ is bounded in $F$ and
equicontinuous in $G'$. Near-closedness puts $x$ in $L$.
Metrizability makes $L$ closed in $F$, and a closed linear subspace
is weakly closed. Hence $G$ is \Ptak.
\end{proof}
This is the classical result of \cite[\S34.3(5)]{Kothe79}; compare
\cite[Corollary 11.1]{Husain62}. In particular
\[
 \omega=\K^{\N},\qquad \omega'_\beta=\varphi=\K^{(\N)}
\]
are both barrelled \Ptak, with product topology on $\omega$ and
the finest locally convex topology on $\varphi$.

\begin{remark}[The scope of Husain's product example]\label{rem:husain}
Husain's Proposition 5 \cite[p.~264]{Husain62} is existential.
Its proof eventually chooses a product index set that is a Hamel
basis of an infinite-dimensional Banach space. Such a basis is
uncountable, by the Baire category theorem. The argument cannot be
specialized to the countable product $\mathbb R^{\N}$ and does not
give a negative conclusion for $(\mathbb R^{\N})'_\beta=\varphi$.

There is a separate omission in the printed Theorem 11 of the same
paper: metrizability is absent from the statement on p.~265, although
the proof uses it on pp.~265--266. The supported formulation assumes
a semireflexive metrizable barrelled space. Corollary 11.1 explicitly
assumes a semireflexive \Frec\ space and is unaffected. Thus the
missing hypothesis and the invalid countable specialization are
different issues; neither invalidates the positive conclusion above.
\end{remark}

\begin{proposition}[\Ptak ness without barrelledness]
\label{prop:nonbarrelled-ptak}
Let $H=\ell^2(\N;\K)$, and let $c$ be the topology on $H'$ of
uniform convergence on norm-compact subsets of $H$. Then
$G=(H',c)$ is \Ptak\ but not barrelled. Consequently, the class
of all Hausdorff \Ptak\ spaces, without barrelledness, satisfies
none of \textup{(O)}, \textup{(C)} and \textup{(G)}.
\end{proposition}
\begin{proof}
The topology $c$ is generated by the seminorms
\[
 p_K(f)=\sup_{x\in K}|f(x)|,\qquad
 K\subseteq H\text{ norm-compact}.
\]
It lies between $\sigma(H',H)$ and the norm topology. Reflexivity
of $H$ therefore identifies $G'$ with $H$.

Let $A\subseteq H=G'$ be a nearly closed linear subspace.
If $x_n\in A$ converges in norm to $x$, then
$K=\{x,x_1,x_2,\ldots\}$ is norm-compact and equicontinuous in $G'$,
by the definition of $c$. Near-closedness makes $A\cap K$
relatively weakly closed, so $x\in A$. Hence $A$ is norm closed,
by metrizability, and therefore $\sigma(H,H')$-closed.
Thus $G$ is \Ptak.

The norm unit ball $B$ of $H'$ is a barrel in $G$: it is closed
because
\[
 B=\bigcap_{\|x\|_H\leq1}\{f\in H':|f(x)|\leq1\}.
\]
Let $e_n'(x)=x_n$ be the coordinate functionals. They have norm
one and converge pointwise to zero on $H$. A finite-net argument,
using the uniform bound on their norms, makes this convergence
uniform on every norm-compact subset of $H$. Hence $2e_n'\to0$
in $c$, whereas $2e_n'\notin B$ for every $n$. The barrel $B$
is not a neighbourhood, so $G$ is not barrelled.

Finally, the identity $(H',\|\cdot\|)\to(H',c)$ is a continuous
linear bijection but is not a topological isomorphism. Both
spaces are \Ptak. This disproves (C), and therefore (O) and (G)
by Theorem~\ref{thm:equivalence}. More explicitly, the inverse
identity has closed graph but is discontinuous.
\end{proof}
The positive assertion is an explicit instance of
\cite[Theorem 6.5]{Ptak58}; that paper also gives a related
failure of the continuous-inverse property on pp.~73--74.
The example isolates the role of barrelledness in the present
classifications, without changing their scope.

\begin{example}[Graph stability is weaker than product stability]
Put $X=\bigoplus_{n\in\N}\ell_2$ and $Y=\prod_{n\in\N}\ell_2$,
using the locally convex direct sum, not the Hilbert sum.
The space $Y$ is reflexive \Frec\ and $Y'_\beta\cong X$.
By Proposition~\ref{prop:dual}, both are barrelled \Ptak.
Nevertheless, $X\times Y$ is not infra-\Ptak, by
\cite[Theorem 2]{Valdivia75}. Thus the two classified classes have
graph stability without finite-product stability. This is why (H)
is the appropriate closure axiom for the proof.
\end{example}

\begin{proposition}[Arbitrary continuous images are too strong]
If a class contains an infinite-dimensional Banach space and is
closed under all injective continuous linear images in Hausdorff
locally convex spaces, then it satisfies none of (O), (C), (G).
\end{proposition}
\begin{proof}
For an infinite-dimensional Banach member $B$, the identity
$(B,\|\cdot\|)\to(B,\sigma(B,B'))$ is a continuous bijection.
Its inverse is not continuous: every basic weak zero-neighbourhood
contains a nonzero finite-codimensional subspace and is unbounded
in norm. Image closure puts both spaces in the class, contradicting (C).
Theorem~\ref{thm:equivalence} then excludes (O) and (G).
This identity is not nearly open, since the norm unit ball is weakly
closed but not a weak zero-neighbourhood.
\end{proof}

\section*{Disclosure of AI assistance}
OpenAI ChatGPT was used as an assistance tool for drafting, checks, and locating related literature.  
The author reviewed, revised, and takes responsibility for the mathematical statements, arguments, citations, and final contents of the manuscript.

\bibliographystyle{alpha}
\bibliography{ptak_classification}
\end{document}